\documentclass{amsart}
\usepackage[margin=1.5in]{geometry}
\usepackage{mathabx}
\usepackage{mathrsfs}
\usepackage{amsfonts}
\usepackage{lineno,hyperref}
\usepackage{amsmath}
\usepackage{amssymb}
\usepackage{amsthm}
\usepackage{cancel}
\usepackage{tikz-cd}
\usepackage{arydshln}
\usepackage{graphicx}
\usepackage{cite}
\usepackage{algorithm}
\usepackage[noend]{algpseudocode}

\newcommand{\RR}{\mathbb{R}}
\newcommand{\CC}{\mathbb{C}}

\newcommand{\ZZ}{\mathbb{Z}}
\newcommand{\QQ}{\mathbb{Q}}
\newcommand{\OO}{\mathcal{O}}
\newcommand{\HH}{\mathbb{H}}
\newcommand{\Dim}{\mathscr{D}}
\newcommand{\Mob}{\text{M\"ob}}

\DeclareMathOperator{\lcm}{lcm}
\DeclareMathOperator{\nrm}{nrm}

\DeclareMathOperator{\disc}{disc}

\DeclareMathOperator{\Mat}{Mat}
\DeclareMathOperator{\Isom}{Isom}
\DeclareMathOperator{\Div}{Div}

\theoremstyle{plain}
\newtheorem{theorem}{Theorem}[section]
\newtheorem{definition}[theorem]{Definition}

\newtheorem{conjecture}{Conjecture}[section]
\newtheorem{lemma}[theorem]{Lemma}

\theoremstyle{remark}
\newtheorem{alg}{Algorithm}[section]

\newtheorem{problem}{Open Problem}
\newtheorem{remark}{Remark}[section]

\begin{document}
\title{Generating Hyperbolic Isometry Groups by Elementary Matrices}

\author{Arseniy (Senia) Sheydvasser}
\address{Department of Mathematics, Technion, Haifa}
\email{sheydvasser@campus.technion.ac.il}

\date{\today}

\begin{abstract}
We consider three families of groups: the Bianchi groups $SL(2,\OO)$ where $\OO$ is the ring of integers of an imaginary, quadratic field; the groups $SL^\ddagger(2,\OO)$ where $\OO$ is a $\ddagger$-order of a definite, rational quaternion algebra with an orthogonal involution; and the groups $SL(2,\OO)$ where $\OO$ is an order of a definite, rational quaternion algebra. We show that such groups are generated by elementary matrices if
and only if $\OO$ is semi-Euclidean (or $\ddagger$-semi-Euclidean), which is a generalization of the usual notion of a Euclidean ring. The proofs are surprisingly simple and proceed by considering fundamental domains of Kleinian groups.
\end{abstract}

\maketitle

\section{Introduction:}\label{section:introduction}

Recently, Matthew Litman and the present author suggested something we termed the Unreasonable Slightness Conjecture\cite{Litman_Sheydvasser_2021}, so named following an earlier paper of Nica\cite{Nica_2011}. A slightly weaker form of it---specifically, one that omits mention of corresponding Apollonian-type packings---reads as follows.

\begin{conjecture}[(Amended) Unreasonable Slightness]\label{conjecture: unreasonable slightness}
Let $\Dim = 3$, $4$, or $5$. Let $\mathcal{A}$ be either an imaginary quadratic field (if $\Dim = 3$), a rational, definite quaternion algebra equipped with an orthogonal involution $\ddagger$ (if $\Dim = 4$), or a rational, definite quaternion algebra (if $\Dim = 5$). Let $\OO$ be a maximal order ($\ddagger$-order, respectively) of $\mathcal{A}$. Let $\Gamma$ either be $SL(2,\OO)$ if $\Dim = 3,5$ or $SL^\ddagger(2,\OO)$ if $\Dim = 4$. Let $\mathcal{E}$ be the subgroup generated by upper and lower triangular matrices. Then exactly one of the following is true.
    \begin{enumerate}
        \item $\OO$ is norm-Euclidean (respectively, norm $\ddagger$-Euclidean, if $\Dim = 4$) and $\Gamma = \mathcal{E}$.
        \item $\OO$ is not Euclidean (respectively, $\ddagger$-Euclidean, if $\Dim = 4)$ and $\mathcal{E}$ is an infinite-index, non-normal subgroup of $\Gamma$.
    \end{enumerate}
\end{conjecture}

\begin{remark}
Proper definitions for $\ddagger$-rings and the groups $SL^\ddagger(2,\OO)$ will be given in Section \ref{section: basic definitions}; for now, the reader should simply think of them as higher-dimensional analogs of orders $\mathfrak{o}_K$ of imaginary quadratic rings and the Bianchi groups $SL(2,\mathfrak{o}_K)$, respectively.
\end{remark}

Our chief goal in the present paper is to prove an even stronger statement, which concerns not just maximal orders (and $\ddagger$-orders), but \emph{all} maximal orders (and $\ddagger$-orders).

\begin{theorem}\label{thm: main theorem}
Let $\Dim = 3$, $4$, or $5$. Let $\mathcal{A}$ be either an imaginary quadratic field (if $\Dim = 3$), a rational, definite quaternion algebra equipped with an orthogonal involution $\ddagger$ (if $\Dim = 4$), or a rational, definite quaternion algebra (if $\Dim = 5$). Let $\OO$ be an order (or $\ddagger$-order) of $\mathcal{A}$. Let $\Gamma$ either be $SL(2,\OO)$ if $\Dim = 3,5$ or $SL^\ddagger(2,\OO)$ if $\Dim = 4$. Let $\mathcal{E}$ be the subgroup generated by upper and lower triangular matrices. Then exactly one of the following is true.
    \begin{enumerate}
        \item $\OO$ is semi-Euclidean (respectively, semi-$\ddagger$-Euclidean, if $\Dim = 4$); $\Gamma = \mathcal{E}$.
        \item $\OO$ is not semi-Euclidean (respectively, semi-$\ddagger$-Euclidean, if $\Dim = 4)$; $\mathcal{E}$ is an infinite-index, non-normal subgroup of $\Gamma$.
    \end{enumerate}
Up to isomorphism, there is only a finite number of semi-Euclidean and semi-$\ddagger$-Euclidean orders---they are enumerated in Tables \ref{tab:semi_euclidean_dim_3}, \ref{tab:semi_euclidean_dim_4}, and \ref{tab:semi_euclidean_dim_5}.
\end{theorem}

\begin{table}
    \centering
    \begin{align*}
        \begin{array}{cc}
         \begin{array}{l|c|c}
              \OO & \text{maximal?} & \text{Euclid?} \\ \hline
              \ZZ[\sqrt{-1}] & \checkmark & \checkmark \\
              \ZZ[\sqrt{-2}]\vphantom{\ZZ\left[\frac{1 + \sqrt{-7}}{2}\right]} & \checkmark & \checkmark \\
              \ZZ\left[\frac{1 + \sqrt{-3}}{2}\right]\vphantom{\ZZ\left[\frac{1 + \sqrt{-7}}{2}\right]} & \checkmark & \checkmark
         \end{array} &
         \begin{array}{l|c|c}
              \OO & \text{maximal?} & \text{Euclid?} \\ \hline
              \ZZ\left[\sqrt{-3}\right] & & \\
              \ZZ\left[\frac{1 + \sqrt{-7}}{2}\right] & \checkmark & \checkmark \\
              \ZZ\left[\frac{1 + \sqrt{-11}}{2}\right] & \checkmark & \checkmark
         \end{array}
         \end{array}
    \end{align*}
    \caption{The semi-Euclidean orders in $\Dim = 3$. If the order is maximal, this is noted with a check mark; similarly, if the order is Euclidean.}
    \label{tab:semi_euclidean_dim_3}
\end{table}

\begin{table}
    \centering
    \begin{align*}
         \begin{array}{c|l|l|c|c}
              H & \substack{\left(\disc(H),\right. \\ \left.\disc(\ddagger)\right)} & \OO & \text{maximal?} & \text{Euclid?} \\ \hline
              \left(\frac{-1,-1}{\QQ}\right) & (2,-1) & \ZZ \oplus \ZZ i \oplus \ZZ j \oplus \ZZ \frac{1 + i + j + ij}{2} & \checkmark & \checkmark \\
              & & \ZZ \oplus \ZZ i \oplus \ZZ j \oplus \ZZ ij & & \checkmark \\ \hdashline
              \left(\frac{-1,-2}{\QQ}\right) & (2,-2) & \ZZ \oplus \ZZ i \oplus \ZZ \frac{1 + i + j}{2} \oplus \ZZ \frac{1 + i + ij}{2} & \checkmark & \checkmark\\
              & & \ZZ \oplus \ZZ i \oplus \ZZ j \oplus \ZZ \frac{j + ij}{2} & &\\
              & & \ZZ \oplus 3\ZZ i \oplus \ZZ \frac{1 + 3i + j}{2} \oplus \ZZ \frac{2i + j + ij}{2} & & \\
              & & \ZZ \oplus \ZZ i \oplus \ZZ j \oplus \ZZ ij & & \\
              & & \ZZ \oplus 2 \ZZ i \oplus \ZZ \frac{1 + i + j}{2} \oplus \ZZ \frac{1 + i - j + 2ij}{2} & & \\
              & & \ZZ \oplus 3\ZZ i \oplus \ZZ \frac{1 + 3i + j}{2} \oplus  \frac{j + 3ij}{2} & & \\ \hdashline
              \left(\frac{-2,-6}{\QQ}\right) & (2,-3) & \ZZ \oplus \ZZ i \oplus \ZZ \frac{i + j}{2} \oplus \ZZ \frac{2 + ij}{4} & \checkmark & \checkmark \\
              & & \ZZ \oplus \ZZ i \oplus \ZZ \frac{i + j}{2} \oplus \ZZ \frac{ij}{2} & & \checkmark \\ \hdashline
              \left(\frac{-2,-3}{\QQ}\right) & (2,-6) & \ZZ \oplus \ZZ i \oplus \ZZ \frac{1 + j}{2} \oplus \ZZ \frac{i + ij}{2} & \checkmark & \checkmark \\ \hdashline
              \left(\frac{-1,-10}{\QQ}\right) & (2,-10) & \ZZ \oplus \ZZ i \oplus \ZZ \frac{1 + i + j}{2} \oplus \ZZ \frac{1 + i + ij}{2} & \checkmark & \checkmark\\ \hdashline
              \left(\frac{-1,-3}{\QQ}\right) & (3,-3) & \ZZ \oplus \ZZ i \oplus \ZZ \frac{i + j}{2} \oplus \ZZ \frac{1 + ij}{2} & \checkmark & \checkmark \\
              & & \ZZ \oplus \ZZ i \oplus \ZZ \frac{1 + j}{2} \oplus \ZZ \frac{i + ij}{2} & \checkmark & \checkmark \\
              & & \ZZ \oplus 2\ZZ i\oplus \ZZ \frac{1 + 2i + j}{2} \oplus \ZZ \frac{i + ij}{2} & & \\
              & & \ZZ \oplus 2\ZZ i \oplus \ZZ \frac{1 + 2i + j}{2} \oplus \ZZ \frac{1 - j + 2ij}{2} & & \\ \hdashline
              \left(\frac{-1,-6}{\QQ}\right) & (3,-6) & \ZZ \oplus \ZZ i \oplus \ZZ \frac{1 + i + j}{2} \oplus \ZZ \frac{1 + i + ij}{2} & \checkmark & \checkmark \\ \hdashline
              \left(\frac{-2,-10}{\QQ}\right) & (5,-5) & \ZZ \oplus \ZZ i \oplus \ZZ \frac{2 + i + j}{4} \oplus \ZZ \frac{2 + 2i + ij}{4} & \checkmark & \checkmark \\ \hdashline
              \left(\frac{-2,-5}{\QQ}\right) & (5,-10) & \ZZ \oplus \ZZ i \oplus \ZZ \frac{1 + i + j}{2} \oplus \ZZ \frac{i + ij}{2} & \checkmark & \checkmark \\ \hdashline
              \left(\frac{-1,-7}{\QQ}\right) & (7,-7) & \ZZ \oplus \ZZ i \oplus \ZZ \frac{i + j}{2} \oplus \ZZ \frac{1 + ij}{2} & \checkmark & \checkmark \\
              & & \ZZ \oplus \ZZ i \oplus \ZZ \frac{1 + j}{2} \oplus \ZZ \frac{i + ij}{2} & \checkmark & \checkmark \\ \hdashline
              \left(\frac{-2,-26}{\QQ}\right) & (13,-13) & \ZZ \oplus \ZZ i \oplus \ZZ \frac{2 + i + j}{4} \oplus \ZZ \frac{2 + 2i + ij}{4} & \checkmark & \checkmark
         \end{array}
    \end{align*}
    \caption{The semi-$\ddagger$-Euclidean orders in $\Dim = 4$. Throughout, $\ddagger$ is chosen such that $(ij)^\ddagger = -ij$. If the order is $\ddagger$-maximal, this is noted with a check mark; similarly, if the order is $\ddagger$-Euclidean.}
    \label{tab:semi_euclidean_dim_4}
\end{table}

\begin{table}
    \centering
    \begin{align*}
         \begin{array}{c|l|l|c|c}
              H & \disc(H) & \OO & \text{maximal?} & \text{Euclid?} \\ \hline
              \left(\frac{-1,-1}{\QQ}\right) & 2 & \ZZ \oplus \ZZ i \oplus \ZZ j \oplus \ZZ \frac{1 + i + j + ij}{2} & \checkmark & \checkmark \\
              & & \ZZ \oplus \ZZ i \oplus \ZZ j \oplus \ZZ ij & & \\
              & & \ZZ \oplus 3\ZZ i \oplus \ZZ (i - j) \oplus \ZZ \frac{1 + i + j + ij}{2} & & \\ \hdashline
              \left(\frac{-3,-1}{\QQ}\right) & 3 & \ZZ \oplus \ZZ \frac{1 + i}{2} \oplus \ZZ j \oplus \ZZ \frac{j + ij}{2} & \checkmark & \checkmark \\
              & & \ZZ \oplus \ZZ i \oplus \ZZ j \oplus \ZZ \frac{1 + i + j + ij}{2} & & \\ \hdashline
              \left(\frac{-2,-5}{\QQ}\right) & 5 & \ZZ \oplus \ZZ i \oplus \ZZ \frac{1 + i + j}{2} \oplus \ZZ \frac{2 + i + ij}{4} & \checkmark & \checkmark
         \end{array}
    \end{align*}
    \caption{The semi-Euclidean orders in $\Dim = 5$. If the order is maximal, this is noted with a check mark; similarly, if the order is Euclidean.}
    \label{tab:semi_euclidean_dim_5}
\end{table}

\begin{remark}
The term ``semi-Euclidean" will be properly defined in Section \ref{section: semi-Euclidean rings}, but the reader should think of it as a slight relaxation of a Euclidean ring such that one can run the Euclidean algorithm on coprime pairs. For the maximal orders and maximal $\ddagger$-orders that we consider, the two notions coincide.
\end{remark}

As we shall see, the proof of this statement for orders other than the ones enumerated in Tables \ref{tab:semi_euclidean_dim_3}, \ref{tab:semi_euclidean_dim_4}, and \ref{tab:semi_euclidean_dim_5} is short and deeply geometric. For the case $\Dim = 3$, Theorem \ref{thm: main theorem} is not new, although the proof is. (However, it is important to note that there was a similar proof of this result in Daniel Martin's 2020 thesis\cite{Martin_2020}, which the author only learned of after the present paper was written.) However, for the cases $\Dim = 4$ and $\Dim = 5$, Theorem \ref{thm: main theorem} is an entirely new result, as we shall discuss below.

This result is a part of a phenomenon that Bogdan Nica termed ``unreasonable slightness,"\cite{Nica_2011} for reasons that will shortly become apparent. Let $K$ be an algebraic number field, $\mathfrak{o}_K$ its ring of integers, $SL(n,\mathfrak{o}_K)$ the special linear group on $\mathfrak{o}_K^n$, and $E(n,\mathfrak{o}_K)$ the subgroup generated by unipotent upper and lower triangular matrices---what are usually called ``elementary" matrices. Then
    \begin{itemize}
        \item if $n > 2$, $SL(n,\mathfrak{o}_K) = E(n,\mathfrak{o}_K)$\cite{Bass_Milnor_Serre_1967},
        \item if $K$ is not an imaginary quadratic field, $SL(n,\mathfrak{o}_K) = E(n, \mathfrak{o}_K)$\cite{Vaserstein_1972},
        \item if $K$ is an imaginary quadratic field, $SL(2,\mathfrak{o}_K) = E(2,\mathfrak{o}_K)$ if and only if $\mathfrak{o}_K$ is a norm-Euclidean ring\cite{Cohn_1966}. Furthermore, if $\mathfrak{o}_K$ is not Euclidean, then $E(2,\mathfrak{o}_K)$ is an infinite-index, non-normal subgroup\cite{Nica_2011,Sheydvasser_2016}.
    \end{itemize}
In short, either the group $SL(n,\mathfrak{o}_K)$ is generated by its elementary matrices, or the group generated by elementary matrices is instead a very meager subgroup. It is worth noting that $E(n,\mathfrak{o}_K)$ can also be characterized as the group generated by all upper and lower triangular matrices, and not simply the unipotent ones; or, if one prefers, it is the subgroup generated by all upper and lower unipotent triangular matrices, as well as all the diagonal ones. Henceforth, we shall take the somewhat nonstandard convention of taking the term ``elementary matrix" to include all upper and lower triangular matrices.

In any case, there are many other results of a similar nature for when $SL(n,R)$---where $R$ is some ring---is generated by upper and lower triangular matrices, and by how many; see, for example, the work of Suslin, Carter and Keller, and Nica\cite{Suslin_1977, Carter_Keller_1983, Nica_2018}. For the specific case where $\OO$ is an order of a quaternion algebra, it is known that if $\OO$ is an $S$-order of a non-split quaternion algebra over an algebraic number field, where $S$ contains at least one non-archimedean valuation, then any element in $SL(n,\OO)$ can be written as a product of at most $2n^2 + 2n - 12$ upper and lower triangular matrices.\cite[Theorem 2.3]{Heald_2007}

The proof for Theorem \ref{thm: main theorem} exploits the fact that the group $\Gamma$ can be understood as an arithmetic subgroup of $\Isom(\HH^{\Dim})$, the isometry group of hyperbolic $\Dim$-space. It proceeds by looking at fundamental domains for the group generated by upper and lower triangular matrices---or, to be a little more accurate, a group commensurable to this one. This point of view is very old, particularly for the Bianchi groups $SL(2,\mathfrak{o}_K)$. However, even for the groups $SL(2,\OO)$, where $\OO$ is an order of a quaternion order, there are other results along these lines: for example, it is known that the number of cusps of the fundamental domain of $SL(2,\OO)$ matches the class number of $\OO$. \cite{Koch_2017}

The simplicity and universality of the proof of Theorem \ref{thm: main theorem} suggest that it might be just a special case of some broader result, which we discuss along with a few other open questions in Section \ref{section: open questions}.

\section{Basic Definitions:}\label{section: basic definitions}

We start by defining all of the terms in the statement of Theorem \ref{thm: main theorem} save for semi-Euclidean rings; most of this discussion is adapted from a paper by Matthew Litman and the current author\cite{Litman_Sheydvasser_2021}, but we repeat it here for convenience. Throughout, we use the standard notation that
    \begin{align*}
        H = \left(\frac{a,b}{\mathbb{F}}\right)
    \end{align*}
is the quaternion algebra over a field $\mathbb{F}$ generated by elements $i,j$ with relations $i^2 = a \in \mathbb{F}^\times$, $j^2 = b \in \mathbb{F}^\times$, and $ij = -ji$; we assume that the characteristic of $\mathbb{F}$ is $0$. In fact, for our purposes, we will always either take $\mathbb{F} = \QQ$, in which case we say that it is rational, or $\mathbb{F} = \RR$. Such a quaternion algebra is called definite if additionally $a < 0$ and $b < 0$, or equivalently if it is a sub-algebra of the classical Hamilton quaternions $H_\RR$.

For any quaternion algebra $H$ over a field $\mathbb{F}$, we can define $SL(n,H)$ in one of several equivalent ways. It is, for instance, the subgroup of $GL(n,H)$ with Dieudonn\'{e} determinant $1$. Perhaps an easier way to think about it is this: take $\overline{\mathbb{F}}$ to be the algebraic closure of $\mathbb{F}$. Then $H' = H \otimes_{\mathbb{F}} \overline{\mathbb{F}} \cong \Mat(2,\overline{\mathbb{F}})$, the ring of $2 \times 2$ matrices over $\overline{\mathbb{F}}$. Therefore, $GL(n,H) \subset GL(n,H') \cong GL(2n,\overline{\mathbb{F}})$. Then $SL(n,H) = GL(n,H) \cap SL(2n,\overline{\mathbb{F}})$, where we have identified $GL(n,H')$ and $GL(2n,\overline{\mathbb{F}})$ by abuse of notation.

Given a ring $R$, an \emph{involution} on $R$ is a group homomorphism $\sigma: R \rightarrow R$ such that $\sigma^2 = id$ and $\sigma(xy) = \sigma(y)\sigma(x)$ for all $x,y \in R$. A homomorphism of rings with involutions $\varphi:(R,\sigma) \rightarrow (S,\sigma')$ is a ring homomorphism such that $\varphi \circ \sigma = \sigma' \circ \varphi$. We shall write $R^+$ to denote the subset of $R$ consisting of elements $\alpha \in R$ such that $\sigma(\alpha) = \alpha$. Rational quaternion algebras admit exactly two kinds of involutions: the standard involution
    \begin{align*}
        \overline{x + yi + zj + tij} = x - yi - zj -tij,
    \end{align*}
and the orthogonal involutions which act as $id$ on a subspace of dimension $3$ and as $-id$ on the orthogonal subspace of dimension $1$. Given an orthogonal involution $\ddagger$ on a quaternion algebra $H$, one can always choose a basis $1,i,j,ij$ for $H$ such that one can express this orthogonal involution as
    \begin{align*}
        \left(x + yi + zj + tij\right)^\ddagger = x + yi + zj - tij,
    \end{align*}
which is the convention that we shall take henceforth.

Any ring with involution $(R,\sigma)$ can be extended to $(\Mat(2,R),\hat{\sigma})$, where
    \begin{align*}
        \hat{\sigma}\left(\begin{pmatrix} a & b \\ c & d \end{pmatrix}\right) = \begin{pmatrix} \sigma(d) & -\sigma(b) \\ -\sigma(c) & \sigma(a) \end{pmatrix},
    \end{align*}
which lets us define the twisted general and special linear groups\cite{Sheydvasser_2019}
    \begin{align*}
        GL^\sigma(2,R) &= \left\{M \in \Mat(2,R)\middle| \hat{\sigma}(M)M = M\hat{\sigma}(M) \in R^\times\right\} \\
        SL^\sigma(2,R) &= \left\{M \in \Mat(2,R)\middle| \hat{\sigma}(M)M = M\hat{\sigma}(M) = 1\right\}.
    \end{align*}
In particular, it is easy to check that $SL^\ddagger(2,H)$ is a subgroup of $SL(2,H)$.

An \emph{order} $\OO$ of a central simple algebra $A$ over a field $\mathbb{F}$ is a (full) lattice that is also a sub-ring. We say that such an order is \emph{Euclidean} if there exists a function---called the \emph{stathm}---$\Phi:R \rightarrow W$ to some well-ordered set $W$ such that for all $a,b \in \OO$ with $b \neq 0$, there exists $q \in \OO$ such that $\Phi(a - bq) < \Phi(b)$. In this case, the usual Euclidean algorithm can be applied to $\OO$; a simple corollary to this is that if $\mathcal{E}$ is the subgroup of $SL(2,\OO)$ generated by upper and lower triangular matrices, then $SL(2,\OO) = \mathcal{E}$. If $\mathbb{F}$ is a number field, we say that $\OO$ is \emph{norm-Euclidean} if one can choose the stathm to be the absolute norm $\OO \rightarrow \ZZ_{\geq 0}$. As an aside, all Euclidean rings that we consider shall be norm-Euclidean, which is somewhat surprising, as there are many Euclidean rings that are not norm-Euclidean: for instance, $\ZZ[\sqrt{14}]$ is such an example\cite{harper_2004}.

A $\sigma$-\emph{order} of a central simple algebra $A$, equipped with an involution $\sigma$, over a field $\mathbb{F}$, is an order $\OO$ which is also a subring with involution---that is, $\sigma(\OO) = \OO$. Such an order is called $\sigma$-\emph{Euclidean}\cite{Sheydvasser_2021} if there exists a function---also referred to as the \emph{stathm}---$\Phi: R\rightarrow W$ to some well-ordered set $W$ such that for all $a,b \in \OO$ with $b \neq 0$ and $a\sigma(b) \in \OO^+$, there exists some $q \in R^+$ such that $\Phi(a - bq) < \Phi(b)$. There is a corresponding $\sigma$-Euclidean algorithm, which has many of the nice properties as the usual Euclidean algorithm; for example, if $(\OO,\sigma)$ is $\sigma$-Euclidean, then $SL^\sigma(2,\OO) = \mathcal{E}$, where $\mathcal{E}$ is the subgroup of $SL^\sigma(2,\OO)$ generated by upper and lower triangular matrices\cite{Sheydvasser_2021}. As above, we say that $\sigma$-order $\OO$ over a number field is \emph{norm $\sigma$-Euclidean} if we can take the stathm to be the restriction of the absolute norm function to $\OO^+$. At the present time, it is an open problem to construct an example of a $\sigma$-Euclidean order that is not norm $\sigma$-Euclidean\cite{Sheydvasser_2021}.

That we must restrict to pairs $a,b$ with $a\sigma(b) \in \OO^+$ might look mysterious, but it is motivated by the observation that
    \begin{align*}
        SL^\sigma(2,R) &= \left\{\begin{pmatrix} a & b \\ c & d \end{pmatrix} \in \Mat(2,R)\middle| a\sigma(d) - b\sigma(c) = 1, \, a\sigma(b) \in R^+, \, c\sigma(d) \in R^+\right\},
    \end{align*}
if $R$ is a subring of a central simple algebra, which is easily checked by computation. Note that if $(a,b)$ is a row of a matrix in $GL^\sigma(2,R)$, then $a\sigma(b) \in R^+$.

All of the above has an interpretation in terms of hyperbolic geometry. Let $\HH^{\Dim}$ denote $\Dim$-dimensional hyperbolic space, represented as the upper half-space of $\RR^{\Dim}$; the boundary $\partial \HH^{\Dim}$ can then be identified with $\RR^{\Dim - 1} \cup \{\infty\}$. We shall write $\Isom(\HH^{\Dim})$ to mean the full isometry group of $\HH^{\Dim}$. It is well-known that $\Isom(\HH^{\Dim})$ is isomorphic to $\Mob(\Dim - 1)$, the group of transformations of $\partial \HH^{\Dim}$ generated by reflections through $(\Dim - 2)$-spheres; correspondingly, the orientation-preserving subgroups are also isomorphic---i.e. $\Isom^0(\HH^{\Dim}) \cong \Mob^0(\Dim - 1)$. This isomorphism is easy to describe: any element in $\Mob(\Dim - 1)$ has a unique extension to an isometry of hyperbolic space, and conversely, any isometry of hyperbolic space yields a unique action on the boundary. Where this relates to our previous algebraic definitions is if $\Dim$ is small. In that case, we have the following accidental isomorphisms:

    \begin{itemize}
        \item $\Mob^0(2) \cong PSL(2,\CC)$ via identifying $\RR^2$ with $\CC$,
        \item $\Mob^0(3) \cong PSL^\ddagger(2,H_\RR)$ via identifying $\RR^3$ with $H_\RR^+$,
        \item $\Mob^0(4) \cong PSL(2,H_\RR)$ via identifying $\RR^4$ with $H_\RR$.
    \end{itemize}
In each case, the action on $\partial \HH^{\Dim}$ is just
    \begin{align*}
        \begin{pmatrix} a & b \\ c & d \end{pmatrix}.\rho = (a\rho + b)(c\rho + d)^{-1}.
    \end{align*}
If we take $\OO$ to be either an order (if $\Dim = 3,5$) or a $\ddagger$-order (if $\Dim = 4$), and either $\Gamma = SL(2,\OO)$ (if $\Dim = 3,5$) or $SL^\ddagger(2,\OO)$ (if $\Dim = 4$), then $\Gamma/\{\pm 1\}$ is discrete subgroup of $\Mob^0(\Dim - 1)$, which is to say that $\Gamma$ is a Kleinian group. However, more than that, $\Gamma$ is an arithmetic group, since it is the set of integer points of an algebraic group. It is well-known that arithmetic groups are lattices which, in this setting, means the following: if one constructs a closed, connected fundamental domain $\mathcal{F}$ for the action of $\Gamma/\{\pm 1\}$ on $\HH^{\Dim}$, then the hyperbolic volume of $\mathcal{F}$ must be finite.

\section{Semi-Euclidean Rings:}\label{section: semi-Euclidean rings}

In the previous section, we defined Euclidean and $\sigma$-Euclidean rings. We now give a simple relaxation of those concepts.

\begin{definition}
Let $R$ be a subring of a division ring. We say that it is \emph{semi-Euclidean} if there a well-ordered set $W$, and a function $\Phi:R \rightarrow W$ called the \emph{stathm} such that for any $a,b \in R$ with $b \neq 0$, if $aR + bR = R$, then there exist $q,r \in R$ such that $a = bq + r$ and $\Phi(r) < \Phi(b)$.
\end{definition}

\begin{definition}
Let $R$ be a subring of a division ring and equipped with an involution $\sigma$. We say that it is \emph{semi-$\sigma$-Euclidean} if there a well-ordered set $W$, and a function $\Phi:R \rightarrow W$ called the \emph{stathm} such that for any $a,b \in R$ with $b \neq 0$, if $aR + bR = R$ and $a\sigma(b) \in R^+$, then there exist $q,r \in R$ such that $a = bq + r$, $q \in R^+$, and $\Phi(r) < \Phi(b)$.
\end{definition}

\begin{remark}
In both cases, since $r \in aR + bR$ and $a \in rR + bR$, it follows that $rR + bR = aR + bR = R$.
\end{remark}

\begin{remark}
To the best of the author's knowledge, these definitions of semi-Euclidean and semi-$\sigma$-Euclidean rings are new. The closest analog might be Campoli's almost Euclidean domains \cite{Campoli_1988}. However, since a ring is an almost Euclidean ring if and only if it is a PID \cite{Green_1997}, the notion of an almost Euclidean ring is strictly stronger.
\end{remark}

The idea behind this relaxation is that if a ring is semi-Euclidean (or semi-$\sigma$-Euclidean), one can still define the Euclidean algorithm as normal, but the inputs have to be restricted to co-prime pairs. This is still enough to prove that $SL(2,R)$ (or, respectively, $SL^\sigma(2,R)$) is generated by elementary matrices, however. For convenience, we will define $\Div(a,b)$ as a function returning the required pair $q,r$.

\begin{alg}\label{alg: semi-Euclidean algorithm}
On an input of $\gamma \in GL(2,R)$ (or $GL^\sigma(2,R)$), this algorithm returns a finite sequence $\gamma_1,\gamma_2,\ldots,\gamma_n$ of upper and lower triangular matrices such that their product is $\gamma$.
\begin{algorithm}[H]
\begin{algorithmic}[1]
\Procedure{SemiEuclideanAlg}{$\gamma$}
\State $l_f \gets []$ \Comment $l_f$ is the list of outputted matrices
\State $\left(\begin{smallmatrix} a & b \\ c & d \end{smallmatrix}\right) \gets \gamma$
\While{$c \neq 0$}
    \If{$\Phi(c) < \Phi(d)$}
        \State $\gamma \gets \gamma\left(\begin{smallmatrix} 0 & 1 \\ 1 & 0 \end{smallmatrix}\right)$
        \State Append $\left(\begin{smallmatrix} 1 & 1 \\ 0 & 1 \end{smallmatrix}\right)$,$\left(\begin{smallmatrix} 1 & 0 \\ -1 & 1 \end{smallmatrix}\right)$,$\left(\begin{smallmatrix} -1 & 1 \\ 0 & 1 \end{smallmatrix}\right)$ to $l_f$
        \State $\left(\begin{smallmatrix} a & b \\ c & d \end{smallmatrix}\right) \gets \gamma$
        \If{$c = 0$}
            \State Exit while-loop
        \EndIf
    \EndIf
    \State $(q,r) \gets \Div(c,d)$
    \State $\gamma \gets \gamma\left(\begin{smallmatrix} 1 & 0 \\ -q & 1 \end{smallmatrix}\right)$
    \State Append $\left(\begin{smallmatrix} 1 & 0 \\ q & 1 \end{smallmatrix}\right)$ to $l_f$
    \State $\left(\begin{smallmatrix} a & b \\ c & d \end{smallmatrix}\right) \gets \gamma$
\EndWhile
\State Prepend $\gamma$ to $l_f$
\State \textbf{return} $l_f$
\EndProcedure
\end{algorithmic}
\end{algorithm}
\end{alg}

\begin{remark}
We note that if either $R$ itself or $\Div$ are not computable, then this is only a semi-algorithm.
\end{remark}

\begin{proof}[Proof of Correctness]
The algorithm is a barely modified version of the standard Euclidean algorithm. It finds upper and lower triangular matrices (in $GL(2,R)$ or $GL^\sigma(2,R)$, as appropriate) to multiply $\gamma$ by on the right, and then appends their inverses to the list $l_f$; once $\gamma$ is upper triangular, it is added to the list, which is then returned. Computing the required triangular matrices works by applying $\Div$ to the lower row $(c,d)$ of a matrix in either $GL(2,R)$ or $GL^\sigma(2,R)$; since it is the lower row, we automatically know that $c R + d R = R$ and furthermore that $c\sigma(d) \in R^+$ if the matrix is in $GL^\sigma(2,R)$, which justifies the application of $\Div$.

It remains to show that this procedure eventually halts, which is to say that eventually, the lower left entry of $\gamma$ is $0$. At each step, if this entry is not $0$, first the lower two entries are switched if necessary by multiplying
    \begin{align*}
        \begin{pmatrix} a & b \\ c & d \end{pmatrix}\begin{pmatrix} 0 & 1 \\ 1 & 0 \end{pmatrix} = \begin{pmatrix} b & a \\ d & c \end{pmatrix},
    \end{align*}
so we may assume that $\Phi(c) \geq \Phi(d)$. Given $(q,r) = \Div(c,d)$,
    \begin{align*}
        \begin{pmatrix} a & b \\ c & d \end{pmatrix}\begin{pmatrix} 1 & 0 \\ -q & 0 \end{pmatrix} &= \begin{pmatrix} * & * \\ c - dq & d \end{pmatrix} = \begin{pmatrix} * & * \\ r & d \end{pmatrix},
    \end{align*}
hence $\Phi(r) < \Phi(d) \leq \Phi(c)$. Thus, the sequence of $\Phi(c_i)$ where $c_i$ is the lower-left entry on the $i$-th step must be strictly decreasing. Since $W$ is well-ordered, this sequence must be finite.
\end{proof}

Although this algorithm only provides a decomposition of lower and upper triangular matrices in $GL(2,R)$ (or $GL^\sigma(2,R)$), it's not hard to see that one can obtain a decomposition inside of $SL(2,R)$ (or $SL^\sigma(2,R)$) simply by conjugating by
    \begin{align*}
        \begin{pmatrix} 1 & 0 \\ 0 & -1 \end{pmatrix}
    \end{align*}
where necessary. This allows us to prove the following.

\begin{theorem}\label{thm: semi-Euclidean classification}
Suppose that closed unit balls centered on $\OO$ (or $\OO^+$, if $\Dim = 4$) cover $\mathcal{A} \otimes_\QQ \RR$ (or $\mathcal{A}^+ \otimes_\QQ \RR$, if $\Dim = 4$). Then $\OO$ is semi-Euclidean and so $\Gamma = \mathcal{E}$. Furthermore, up to isomorphism there are only finitely many such orders, and they are enumerated in Tables \ref{tab:semi_euclidean_dim_3}, \ref{tab:semi_euclidean_dim_4}, and \ref{tab:semi_euclidean_dim_5}.
\end{theorem}

\begin{proof}
Write $\Lambda$ for either $\OO$ (if $\Dim = 3,5$) or $\OO^+$ (if $\Dim = 4$). The condition that closed unit balls centered on $\Lambda$ cover the ambient space is just saying that the covering radius $\mu(\Lambda)$ of the lattice $\Lambda$ is no more than $1$. If it is less than $1$, then $\OO$ is Euclidean (or $\ddagger$-Euclidean) simply by taking the stathm $\Phi$ to be the norm. All such orders have previously been enumerated \cite{Vigneras_1980,Sheydvasser_2021}, so we are primarily interested in the case where $\mu(\Lambda) = 1$. In $\Dim = 3$, there is only one such order, namely $\ZZ[\sqrt{-3}]$. This order is semi-Euclidean if we take the stathm $\Phi$ to be the norm; this is easy to see directly, although it also can be seen as a consequence of a theorem of R. Keith Dennis\cite{Dennis_1975}.

For the other two cases, we make use of the theory of successive minima. Recall that given a lattice $\Lambda$, its $i$-th successive minimum is the infimum of all $r > 0$ such that $\Lambda$ contains $i$ linearly independent vectors of length no more than $r$. It is easy to check that $\mu(\Lambda) \geq \lambda_{\Dim - 1}(\Lambda)/2$, hence $\lambda_{\Dim - 1}(\Lambda) \leq 2$. By Minkowski's second theorem,
    \begin{align*}
        \lambda_1(\Lambda)\ldots\lambda_{\Dim - 1}(\Lambda)\text{vol}(B_1) \geq \frac{2^{\Dim - 1}}{(\Dim - 1)!}\text{vol}\left(\RR^{\Dim - 1}/\Gamma\right),
    \end{align*}
where $B_1$ is the unit ball in $\RR^{\Dim - 1}$---this has volume $4\pi/3$ if $\Dim = 4$ and $\pi^2/2$ if $\Dim = 5$. Note that $1 = \lambda_1(\Lambda) \leq \ldots \leq \lambda_{\Dim - 1}(\Lambda)$, so we get a bound on the co-volume of $\Gamma$:
    \begin{align*}
        \text{vol}\left(\RR^{\Dim - 1}/\Gamma\right) &\leq \begin{cases} 4\pi & \text{if $\Dim = 4$} \\ 6\pi^2 & \text{if $\Dim = 5$} \end{cases} < \begin{cases} 13 & \text{if $\Dim = 4$} \\ 60 & \text{if $\Dim = 5$} \end{cases}.
    \end{align*}
If $\Dim = 5$, the co-volume of $\Gamma$ is $[\OO':\Gamma]\text{vol}\left(\RR^4/\OO'\right)$ where $\OO'$ is a maximal order containing $\OO$. The co-volume of a maximal order of a rational, definite quaternion algebra $H$ is just twice the discriminant of $H$. Note that $\mu(\OO') \leq \mu(\OO)$, so we can start by looking at all isomorphism classes of maximal orders of rational, definite quaternion algebras with discriminant less than $120$---there are only finitely many of these, and they can be easily enumerated---and then enumerating all sub-orders with index no more than $60/\sqrt{|\disc(H)|}$, and with a basis of elements with norm no more than $2$ (since $\lambda_{\Dim - 1}(\Lambda) \leq 2$). For each of these, we can compute the packing radius and check whether it is greater than $1$ or not. In principle, this can be done by hand, but it is vastly easier to use, say, Magma to do it automatically, producing Table \ref{tab:semi_euclidean_dim_5}. Three of the orders are maximal and Euclidean; the other three are non-maximal and have covering radius $1$. For each of these orders, we can compute elements $\alpha$ such that:
    \begin{enumerate}
        \item the norm of $\alpha$ is $1$,
        \item $\alpha$ is distance $1$ away from $\OO$, and
        \item if $\beta$ is distance $1$ away from $\OO$, then it is a translate of one of the $\alpha$'s by an element in $\OO$.
    \end{enumerate}
These elements are listed in Table \ref{tab:deep_holes_dim_5}, together with maximal orders $\OO'$ such that $\alpha \in \OO' \supset \OO$. We shall refer to these elements as \emph{holes}---they are (up to translation) the unique elements not covered by unit disks centered at points in $\OO$. As such, they are the possible obstructions to the rings $\OO$ being semi-Euclidean with respect to the norm.

Now, choose any $a,b \in \OO$ such that $a\OO + b\OO = \OO$ and $b \neq 0$, and determine the distance from $b^{-1}a$ to $\OO$. Since the covering radius is $1$, this distance is no more than $1$. If it is exactly $1$, then $b^{-1}a = \alpha + \OO$ where $\alpha$ is a hole. This shows that $a \in b(\alpha + \OO) \subset b\OO'$, showing that $\OO = a\OO + b\OO \subset b\OO'$. This can occur only if $\nrm(b) = 1$, in which case we write $a = bq + r$ with $q = b^{-1}a$ and $r = 0$. Since $|r| = 0 < |b| = 1$, we are done. If the distance from $b^{-1}a$ to $\OO$ is less than $1$, then the argument is even easier: take $q$ to be the closest element to $b^{-1}a$ and write $a = bq + r$ where $r = a - bq$. Since $|b^{-1}a - q| < 1$, we know that $|r| < |b|$. Therefore, we have proved that all of these rings are semi-Euclidean. The ones with covering radius equal to $1$ are not Euclidean since it is known that an order $\OO$ that is Euclidean must be maximal\cite{Cerri_Chabert_Pierre_2013}.

\begin{table}
    \centering
    \begin{align*}
         \begin{array}{c|l|l|l}
            H & \OO & \alpha & \OO' \\ \hline
              \left(\frac{-1,-1}{\QQ}\right) & \ZZ \oplus \ZZ i \oplus \ZZ j \oplus \ZZ ij & \frac{1 + i + j + k}{2} & \ZZ \oplus \ZZ i \oplus \ZZ j \oplus \ZZ \frac{1 + i  + j + ij}{2} \\ \hdashline
              & \ZZ \oplus 3\ZZ i \oplus \ZZ (i - j) \oplus \ZZ \frac{1 + i + j + ij}{2} & -ij & \ZZ \oplus \ZZ i \oplus \ZZ j \oplus \ZZ \frac{1 + i + j + ij}{2} \\
              & & \frac{1 + i -j - ij}{2} & \ZZ \oplus \ZZ i \oplus \ZZ j \oplus \ZZ \frac{1 + i + j + ij}{2} \\
              & & \frac{4i + j + ij}{3} & \ZZ \oplus 3\ZZ i \oplus \ZZ (i - j) \oplus \ZZ \frac{3 - 5i + j + ij}{6} \\
              & & \frac{2i + 2j - ij}{3} & \ZZ \oplus 3\ZZ i \oplus \ZZ (i - j) \oplus \ZZ \frac{3 + i + j - 5ij}{6} \\ \hdashline
              \left(\frac{-3,-1}{\QQ}\right) & \ZZ \oplus \ZZ i \oplus \ZZ j \oplus \ZZ \frac{1 + i + j + ij}{2} & \frac{1 + ij}{2} & \ZZ \oplus \ZZ i \oplus \ZZ \frac{i + j}{2} \oplus \ZZ \frac{1 + ij}{2} \\
              & & \frac{j + ij}{2} & \ZZ \oplus \ZZ \frac{1 + i}{2} \oplus \ZZ j \oplus \ZZ \frac{j + ij}{2}
        \end{array} 
    \end{align*}
    \caption{For $\Dim = 5$, equivalence classes of maximally distant elements $\alpha$ in $\OO$, together with maximal orders $\OO'$ that contain them.}
    \label{tab:deep_holes_dim_5}
\end{table}

The proof for the $\Dim = 4$ case is similar, although a little more involved. By Minkowski's second theorem, we have a bound on $\text{vol}(\RR^3/\Lambda)$, but what we want is a bound on $\text{vol}(\RR^4/\OO)$. To get this, note that since $\lambda_3(\Lambda) \leq 2$, this implies that $\lambda_4(\Lambda) \leq 4$, since we can always take two of these three short vectors and multiply them to get another short, linearly independent, vector. Therefore,
    \begin{align*}
        \text{vol}(\RR^4/\OO') \leq \text{vol}(\RR^4/\OO) \leq 4\text{vol}(\RR^3/\Lambda) \leq 16\pi.
    \end{align*}
where $\OO'$ is a maximal $\ddagger$-order containing $\OO$. But
    \begin{align*}
        \text{vol}(\RR^4/\OO') = \frac{\lcm\left(|\disc(H)|,\nrm(\xi)\right)}{4}
    \end{align*}
using the computation of $16\text{vol}(\RR^4/\OO')^2$ done previously by the author\cite{Sheydvasser_2017}. Therefore, we only need to consider the finitely many definite, rational quaternion algebras $H$ with discriminant no more than $64\pi < 202$, and all orthogonal involutions $\ddagger$ such that $\nrm(\xi) < 202$. The maximal $\ddagger$-orders $\OO'$ of any such quaternion algebra with involution will be orders of the form $\OO'' \cap {\OO''}^\ddagger$ with discriminant $|\disc(H)|\nrm(\xi)$, where $\OO''$ is a maximal order of $H$\cite{Sheydvasser_2017}---there are only finitely many of these as well. Finally, any order $\OO$ with covering radius $1$ must be a sub-$\ddagger$-order of such an $\OO'$ possessing a basis of elements in $\OO'$ of length no more than $4$---once again, there are only finitely many choices. Thus, we see that all possible candidates $\OO$ can be enumerated; the results are given in Table \ref{tab:semi_euclidean_dim_4}.

It remains to check that these orders are semi-$\ddagger$-Euclidean---the argument is essentially the same as for the $\Dim = 5$ case: to start, we compute the covering radius for $\OO^+$. If it is less than $1$, then we know that $\OO$ is $\ddagger$-Euclidean. If it is exactly $1$, then we compute holes $u \in H$ that are distance $1$ from $\OO^+$ and note that in each case there exists a maximal $\ddagger$-order $\OO'$ such that $u \in \OO' \supset \OO$; these are collected in Table \ref{tab:deep_holes_dim_4}. These maximal $\ddagger$-orders $\OO'$ are $\ddagger$-Euclidean, and therefore $\OO$ is semi-$\ddagger$-Euclidean. The $\ddagger$-orders $\OO$ are not themselves $\ddagger$-Euclidean as it is known that $\OO$ is $\ddagger$-Euclidean only if $\OO^+ = \OO'^+$ for some maximal $\ddagger$-order $\OO'$ \cite{Sheydvasser_2021}.
\end{proof}

\begin{table}
    \centering
    \begin{align*}
         \begin{array}{c|l|l|l}
            H & \OO & \alpha & \OO' \\ \hline
              \left(\frac{-1,-2}{\QQ}\right) & \ZZ \oplus \ZZ i \oplus \ZZ j \oplus \ZZ \frac{j + ij}{2} & \frac{1 + i + j}{2} & \ZZ \oplus \ZZ i \oplus \ZZ \frac{1 + i + j}{2} \oplus \ZZ \frac{1 + i + ij}{2} \\
              & \ZZ \oplus 3\ZZ i \oplus \ZZ \frac{1 + 3i + j}{2} \oplus \ZZ \frac{2i + j + ij}{2} & \pm i & \\
              & \ZZ \oplus \ZZ i \oplus \ZZ j \oplus \ZZ ij & \frac{1 + i + j}{2} &  \\
              & \ZZ \oplus 2 \ZZ i \oplus \ZZ \frac{1 + i + j}{2} \oplus \ZZ \frac{1 + i - j + 2ij}{2} & i & \\
              & \ZZ \oplus 3\ZZ i \oplus \ZZ \frac{1 + 3i+j}{2} \oplus  \frac{j + 3ij}{2} & \pm i & \\ \hdashline
              \left(\frac{-1,-3}{\QQ}\right) & \ZZ \oplus 2\ZZ i\oplus \ZZ \frac{1 + 2i + j}{2} \oplus \ZZ \frac{i + ij}{2} & i & \ZZ \oplus \ZZ i \oplus \ZZ \frac{1 + j}{2} \oplus \ZZ \frac{i + ij}{2} \\
              & \ZZ \oplus 2\ZZ i \oplus \ZZ \frac{1 + 2i + j}{2} \oplus \ZZ \frac{1 - j + 2ij}{2} & i &
        \end{array} 
    \end{align*}
    \caption{For $\Dim = 4$, equivalence classes of maximally distant elements $\alpha$ in $\OO^+$, together with maximal $\ddagger$-orders $\OO'$ that contain them.}
    \label{tab:deep_holes_dim_4}
\end{table}

\section{Construction of Fundamental Domains:}

We have shown that if the covering radius of $\OO$ (or $\OO^+$) is no more than $1$, then $\OO$ is semi-Euclidean (or semi-$\ddagger$-Euclidean). It remains for us to show that if the covering radius is larger than $1$, then the group $\mathcal{E}$ generated by elementary matrices is infinite index in $G$---equivalently, that $\mathcal{E}$ is not a lattice. Our approach is to first replace $\mathcal{E}$ with a commensurable group, which we shall do in stages. Throughout, we shall write $\Lambda$ for either $\OO$ (if $\Dim = 3,5$) or $\OO^+$ (if $\Dim = 4$).

\begin{definition}
Let $W$ be the subgroup of $\Isom(\HH^{\Dim})$ generated by all elements
    \begin{align*}
        E_\alpha:\partial\HH^{\Dim} &\rightarrow \partial\HH^{\Dim} \\
        z &\mapsto \begin{pmatrix} \alpha & 1 \\ -1 & 0 \end{pmatrix}.z = -z^{-1} - \alpha
    \end{align*}
with $\alpha \in \Lambda$.
\end{definition}

The group $W$ is commensurable with $\mathcal{E}$---specifically, any element in $\mathcal{E}$ can be written in the form
    \begin{align*}
        \begin{pmatrix} u & 0 \\ 0 & v \end{pmatrix}\gamma
    \end{align*}
for some $u,v \in \OO^\times$ and $\gamma \in W$ \cite{Litman_Sheydvasser_2021}---therefore, $W$ is a subgroup of $\mathcal{E}$ of index at most $|\OO^\times|^2/4$. However, we will need a slightly larger subgroup for our purposes.

\begin{definition}
Let $H$ be the subgroup of $\Isom(\HH^{\Dim})$ generated by all elements
    \begin{align*}
        E_\alpha: \partial \HH^{\Dim} &\rightarrow \partial \HH^{\Dim} \\
        z & \mapsto \begin{pmatrix} \alpha & 1 \\ -1 & 0 \end{pmatrix}.z =-z^{-1} - \alpha
    \end{align*}
with $\alpha \in \Lambda$ and by
    \begin{align*}
        \psi: \partial \HH^{\Dim} &\rightarrow \partial \HH^{\Dim} \\
        z &\mapsto -\overline{z}.
    \end{align*}
Similarly, let $K$ be the subgroup of $\Isom(\HH^{\Dim})$ generated by all elements
    \begin{align*}
        \phi_\alpha: \partial \HH^{\Dim} &\rightarrow \partial \HH^{\Dim} \\
        z &\mapsto \overline{(z - \alpha)}^{-1} + \alpha
    \end{align*}
and
    \begin{align*}
        T_\alpha: \partial \HH^{\Dim} &\rightarrow \partial \HH^{\Dim} \\
        z &\mapsto z + \alpha.
    \end{align*}
\end{definition}

\begin{lemma}\label{essential relations}
For all $\alpha,\beta \in \Lambda$, the following relations hold.
    \begin{align*}
        T_\alpha \circ \phi_\beta &= \phi_{\alpha + \beta} \circ T_\alpha \\
        \psi \circ E_\alpha &= E_{-\overline{\alpha}} \circ \psi \\
        &=  \phi_{\overline{\alpha}} \circ T_{\overline{\alpha}}\\
        \psi \circ \phi_\alpha &= \phi_{-\overline{\alpha}} \circ \psi \\
        &= E_{\overline{\alpha}} \circ T_{-\alpha} \\
        \psi \circ T_\alpha &= T_{-\overline{\alpha}} \circ \psi.
    \end{align*}
\end{lemma}

\begin{proof}
This is basic algebra.
\end{proof}

\begin{lemma}
Both $W$ and $K$ are index two subgroups of $H$.
\end{lemma}

\begin{proof}
That $W$ is a subgroup of $H$ is obvious. That $K$ is also a subgroup follows from the observation that any generator $E_\alpha = \psi \circ \phi_{\overline{\alpha}} \circ T_{\overline{\alpha}}$ per Lemma \ref{essential relations}, and so $H$ can equivalently be considered as the group generated by elements $\psi$, $\phi_\alpha$, and $T_\alpha$. Furthermore, using the relations in Lemma \ref{essential relations}, it follows that any element $\gamma \in H$ can be written in the form
    \begin{align*}
        \gamma = \psi^k \circ E_{\alpha_1} \circ E_{\alpha_2} \circ \ldots \circ E_{\alpha_n}
    \end{align*}
for some $k \in \{0,1\}$ and $\alpha_1,\ldots,\alpha_n \in \Lambda$, or in the form
    \begin{align*}
        \gamma = \psi^k \circ \phi_{\alpha_1} \circ \phi_{\alpha_2} \circ \ldots \circ \phi_{\alpha_n} \circ T_\beta
    \end{align*}
for some $k \in \{0,1\}$ and $\alpha_1,\ldots,\alpha_n,\beta \in \Lambda$. The first form proves that $W$ is an index two subgroup of $H$; the second form proves that $K$ is an index two subgroup of $H$.
\end{proof}

\begin{lemma}\label{dihedral angles}
If $S_1, S_2$ are intersecting unit spheres in $\RR^{\Dim - 1}$ centered at points in $\Lambda$, then the dihedral angle between them is either $0$, $\pi/3$, or $\pi/2$.
\end{lemma}

\begin{proof}
Let $z_1, z_2$ be the centers of $S_1,S_2$. Then $\nrm(z_1 - z_2) \in \ZZ$, which means that the centers of $S_1$ and $S_2$ can be distance $1$, $\sqrt{2}$, $\sqrt{3}$, or $2$ apart. It is an easy computation that the corresponding dihedral angles are $\pi/3$, $\pi/2$, $\pi/3$, and $0$.
\end{proof}

\begin{theorem}\label{thm: fundamental domain description}
The group $K$ is a lattice if and only if the covering radius of $\Lambda$ is no more than $1$. Moreover, $K$ always admits a convex fundamental polyhedron with finitely many faces---consequently, it is always geometrically finite.
\end{theorem}

\begin{remark}
Three examples of such convex fundamental polyhedra are shown in Figure \ref{fig:fundamental_domains}. The corresponding lattices $\Lambda$ and their covering radii are shown in Figure \ref{fig:covering_radii}.
\end{remark}

\begin{figure}
    \centering
    \begin{tabular}{ccc}
    \includegraphics[height = 0.25\textwidth]{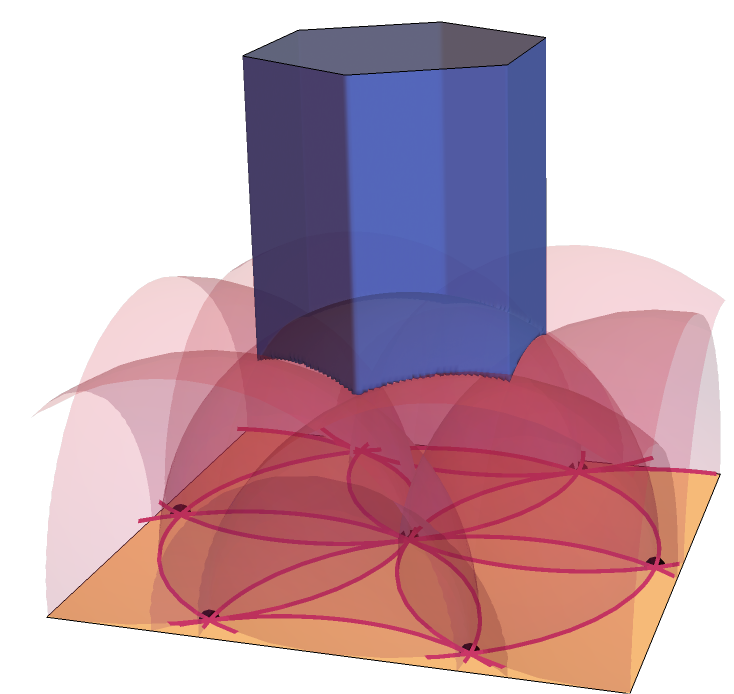} & \includegraphics[height = 0.25\textwidth]{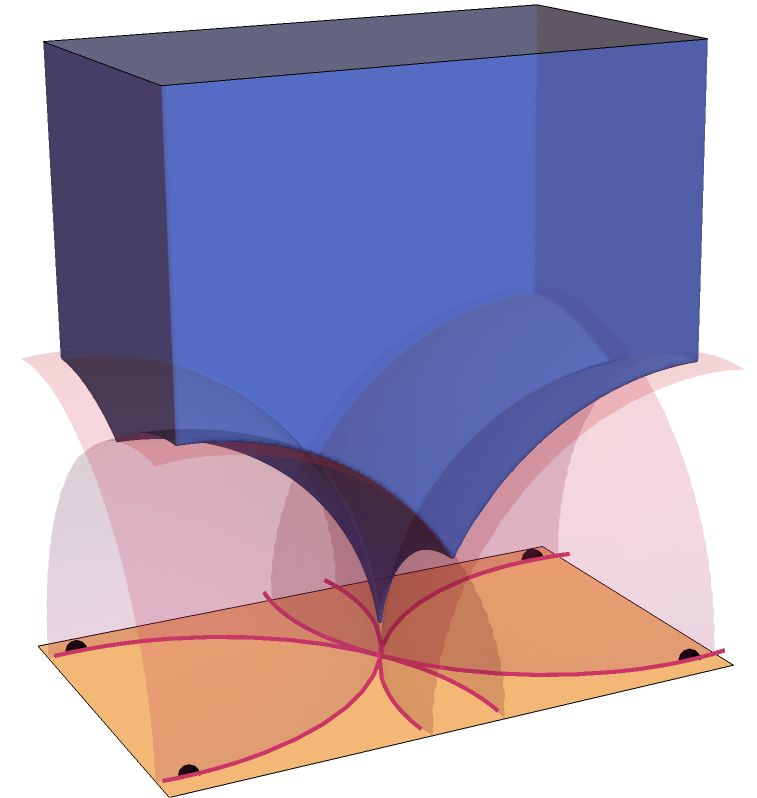} & \includegraphics[height = 0.25\textwidth]{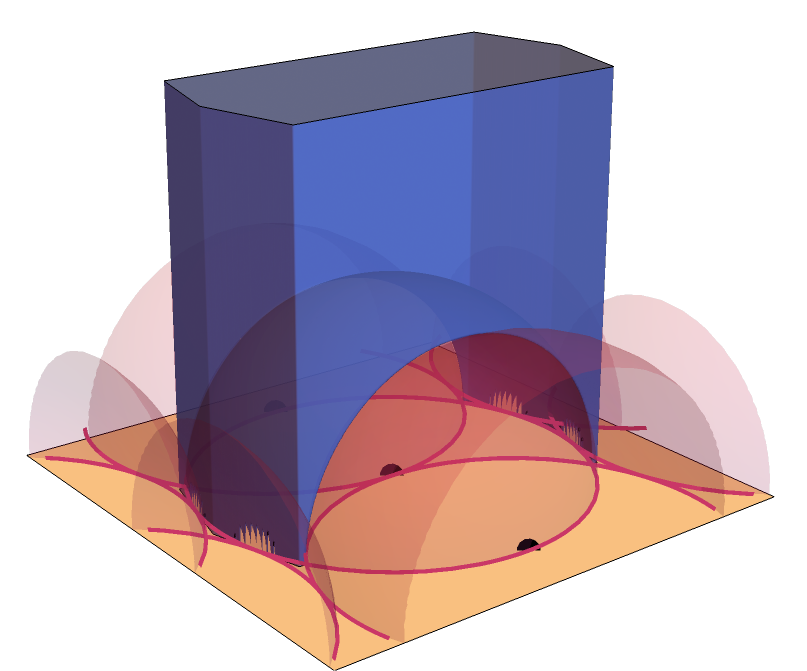}
    \end{tabular}
    \caption{In blue, the fundamental domains for the groups $K$ corresponding to $\ZZ[(1 + \sqrt{-3})/2]$, $\ZZ[\sqrt{-3}]$, and $\ZZ[(1 + \sqrt{-15})/2]$. Unit balls centered at points in $\Lambda$ are drawn in red. The first two groups are lattices; the last one is not.}
    \label{fig:fundamental_domains}
\end{figure}

\begin{proof}
We construct a fundamental domain for $K$ as follows: first, consider the action of the translations $T_\alpha$ on $\RR^{\Dim - 1}$. Since $\Lambda$ is a lattice in $\RR^{\Dim - 1}$, this group of translations admits a fundamental domain $\mathcal{P}$, which we can take to be the interior of some polygon. The details of how we choose to do this are irrelevant---we can take it to be the Dirichlet domain, for example. Then, if we take $\mathcal{P} \times (0,\infty) \subset \HH^{\Dim}$, this is a fundamental domain for this translation group in $\HH^{\Dim}$.

Next, consider the action of the elements $\phi_\alpha$. It is easily checked that these are reflections through unit spheres centered at $\alpha$. Furthermore, these spheres intersect at dihedral angles which are of the form $\pi/k$ for some $k \in \ZZ$, as per Lemma \ref{dihedral angles}. Consequently, the group generated by these reflections is a geometric reflection group, and so it admits a very simple fundamental domain $\mathcal{R} \subset \HH^{\Dim}$---simply take the set of points in $\HH^{\Dim}$ which are in the exterior of all of the aforementioned unit spheres. I claim that 
    \begin{align*}
        \mathcal{F} = \mathcal{R} \cap \left(\mathcal{P} \times (0,\infty)\right)
    \end{align*}
is a fundamental domain for $K$. Here is why: first, note that any point $\rho \in \HH^{\Dim}$ can be moved into the closure of $\mathcal{F}$ by first using the reflection $\phi_\alpha$ to move it into $\mathcal{R}$---but then, one can use translations $T_\alpha$ to move it into $\mathcal{F}$. On the other hand, we need to prove that if $\rho \in \mathcal{F}$, then there does not exist any non-identity element $\gamma \in K$ such that $\gamma(\rho) \in \mathcal{F}$. By Lemma \ref{essential relations}, we know that
    \begin{align*}
        \gamma = \phi_{\alpha_1} \circ \ldots \circ \phi_{\alpha_n} \circ T_{\beta}
    \end{align*}
for some $\alpha_1,\ldots, \alpha_n, \beta \in \Lambda$. If $\gamma$ can be written without reflection, then we use the fact that $\mathcal{P} \times (0,\infty)$ is a fundamental domain for the translation group to conclude that $T_\beta(\rho) = \rho$ if and only if $T_\beta$ is the identity. Otherwise, we note that $T_\beta$ will move $\rho$ to some other point in $\mathcal{R}$, and then the reflections must necessarily move it out of $\mathcal{R}$, as $\mathcal{R}$ is a fundamental domain for the reflection group.

But now, notice that this fundamental domain $\mathcal{F}$ has finite volume in $\HH^{\Dim}$ if and only if the closed unit balls centered at points in $\Lambda$ cover $\RR^{\Dim - 1}$. That is, $K$ is a lattice if and only if the covering radius of $\Lambda$ is no more than $1$. On the other hand, it always admits a convex fundamental polyhedron with finitely many faces.
\end{proof}

\begin{figure}
    \centering
    \begin{tabular}{ccc}
    \includegraphics[height = 0.2\textwidth]{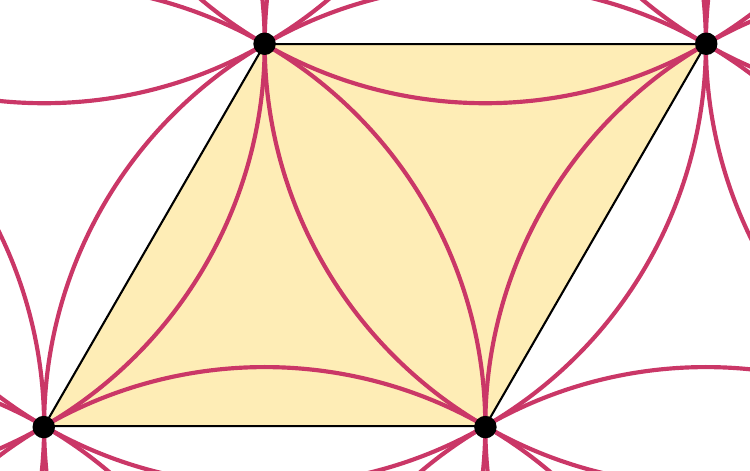} & \includegraphics[height = 0.35\textwidth]{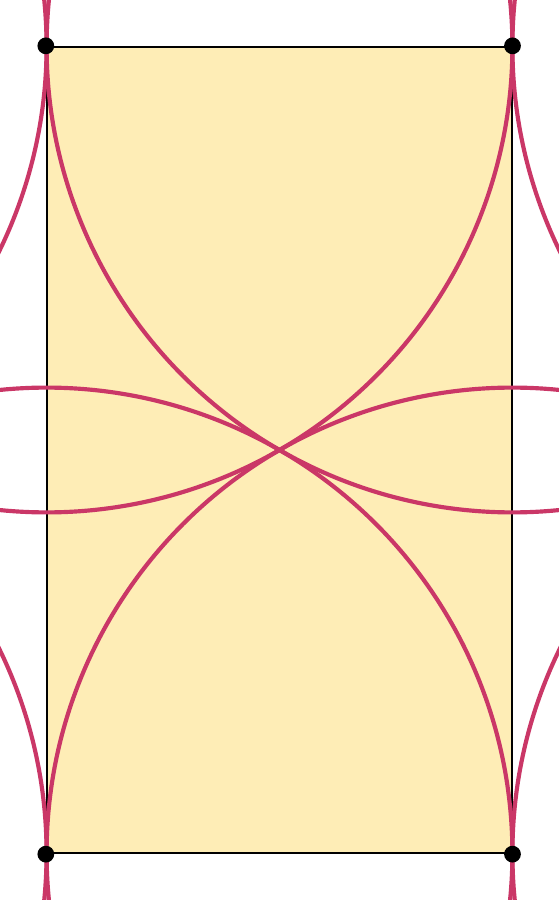} & \includegraphics[height = 0.35\textwidth]{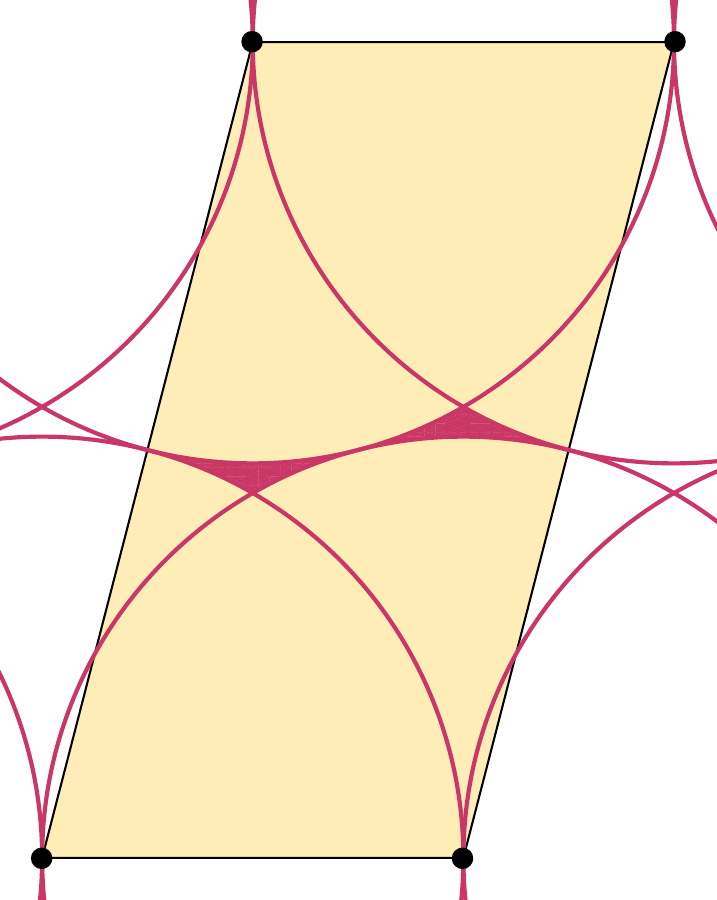}
    \end{tabular}
    \caption{The fundamental parallelograms for $\ZZ[(1 + \sqrt{-3})/2]$, $\ZZ[\sqrt{-3}]$, and $\ZZ[(1 + \sqrt{-15})/2]$. Unit circles centered at points in $\Lambda$ are drawn in red, as is the area not covered by such disks. The covering radius of the first order is less than $1$; the covering radius of the second order is $1$; the covering radius of the last order is more than $1$.}
    \label{fig:covering_radii}
\end{figure}

\begin{remark}
Our approach for the proof of this theorem is directly inspired by---and is very similar to---the construction of fundamental domains for Apollonian-like groups in a paper by Matt Litman and the present author\cite[Theorem 5.2]{Litman_Sheydvasser_2021}.
\end{remark}

This is sufficient for us to prove our main theorem.

\begin{proof}[Proof of Theorem \ref{thm: main theorem}]
For any $\OO$, if the covering radius of $\Lambda$ is no more than $1$, then by Theorem \ref{thm: semi-Euclidean classification} $\OO$ is semi-Euclidean (or semi-$\ddagger$-Euclidean if $\Dim = 4$) as described in Tables \ref{tab:semi_euclidean_dim_3}, \ref{tab:semi_euclidean_dim_4}, and \ref{tab:semi_euclidean_dim_5}; by the existence of Algorithm \ref{alg: semi-Euclidean algorithm}, we know that this implies that $\mathcal{E} = \Gamma$. If the covering radius is more than $1$, then by Theorem \ref{thm: fundamental domain description}, $K$ is not a lattice. But $K$ is commensurable to $\mathcal{E}$, so $\mathcal{E}$ is not a lattice either, which is to say that it is an infinite-index subgroup of $\Gamma$; again, by the existence of Algorithm \ref{alg: semi-Euclidean algorithm}, this proves that $\OO$ is not semi-Euclidean (or semi-$\ddagger$-Euclidean, if $\Dim = 4$).

It remains to prove that if $\mathcal{E}$ is infinite-index, then it is also non-normal. Suppose that it is a normal subgroup of $\Gamma$. Then $\Gamma$ acts on the limit set of $\mathcal{E}$, which implies that its limit set is all of $\partial \HH^{\Dim}$. Since $K$ is commensurable to $\mathcal{E}$, its limit set is also $\partial \HH^{\Dim}$. However, this is impossible: it isn't hard to see that the limit set of $K$ must be contained in the subset of $\partial \HH^{\Dim}$ covered by unit balls centered on points in $\Lambda$; if $\mathcal{E}$ is infinite-index, we know that that this is a proper subset of $\partial \HH^{\Dim}$.
\end{proof}

\section{Open Questions and Discussion:}\label{section: open questions}

Theorem \ref{thm: main theorem} raises many new questions, which we'll now discuss in turn. We begin with the broadest, big picture question.

\begin{problem}
Can Theorem \ref{thm: main theorem} be generalized to any hyperbolic space $\HH^{\Dim}$, and not just for the special cases $\Dim = 3,4,5$? Or, in other words, is unreasonable slightness a general phenomenon?
\end{problem}

\begin{remark}
There is an obvious---if somewhat disappointing---generalization of the result for $\Dim = 2$: since $\Isom^0(\HH^2) \cong PSL(2,\RR)$, we observe that there is only one discrete subring of $\RR$, namely $\ZZ$. Of course, $\ZZ$ is Euclidean and $SL(2,\ZZ) = E(2,\ZZ)$.

It is not clear what can be said for $\Dim > 5$. The accidental isomorphisms that we used to define the problem do not generalize (they are, after all, accidental), and of course, there are no real division algebras larger than $H_\RR$. On the other hand, perhaps trying to find discrete sub-rings of real division algebras is too constrained a viewpoint; perhaps instead one should study lattices of $\RR^{\Dim - 1}$ with some weaker algebraic structure. The geometric construction of the fundamental domain in Theorem \ref{thm: fundamental domain description} can almost certainly be generalized to something of this type; the bigger question is what should be the correct relaxation of the notion of a semi-Euclidean ring?

To sum up, what we are positing is the existence of some kind of algebraic structure $\mathcal{A}$ on $\RR^{\Dim - 1}$---reducing to the normal ring structure for $\Dim \leq 5$---such that $\Isom^0(\HH^{\Dim})$ is isomorphic to something like $PSL(\RR^{\Dim - 1})$ and a ``Euclidean-like" property $P$ so that for all discrete substructures $\Lambda$ of $\mathcal{A}$, $\Lambda$ has property $P$ if and only if ``$PSL(\Lambda)$" is generated by elementary matrices---moreover, if it doesn't have property $P$, then the group generated by elementary matrices is infinite-index.
\end{remark}

\begin{problem}
Let $H$ be a quaternion algebra over an algebraic number field $K \neq \QQ$ and let $\ddagger$ be an orthogonal involution on $H$. Does there exist a maximal order $\OO$ of $H$ such that $SL(n,\OO)$ is not generated by elementary matrices for some $n$? Does there exist a maximal $\ddagger$-order $\OO$ such that $SL^\ddagger(2,\OO)$ is not generated by elementary matrices? Is it possible to define $SL^\ddagger(n,\OO)$ in some natural way---if it is, is it necessarily generated by elementary matrices? If we repeat this analysis with $K = \QQ$ but consider $H$ indefinite, are there any examples where these groups are not generated by elementary matrices?
\end{problem}

\begin{remark}
This question is the natural generalization of the observations made about the structure of $SL(n,\mathfrak{o}_K)$ in Section \ref{section:introduction} to the $\Dim = 4$ and $\Dim = 5$ cases. One possible method to attack this question would be to try to generalize Vaserstein's arguments for $SL(2,\mathfrak{o}_K)$; however, this would require a generalization of the Mennicke symbol that would make sense for noncommutative rings and rings with involutions.
\end{remark}

\begin{problem}
Do all semi-Euclidean and semi-$\ddagger$-Euclidean rings have left class number $1$? Are they always Gorenstein? What other nice properties do they have?
\end{problem}

\begin{remark}
It is an easy calculation to check that all semi-Euclidean orders enumerated in Theorem \ref{thm: main theorem} have left class number $1$ and are Gorenstein orders, but there is no obvious reason why this should be true in general.
\end{remark}

\begin{remark}
One nice property that semi-Euclidean and semi-$\ddagger$-Euclidean rings certainly do not have in general is hereditary-ness---in fact, the Lipshitz order
    \begin{align*}
        \ZZ \oplus \ZZ i \oplus \ZZ j \oplus \ZZ ij \subset \left(\frac{-1,-1}{\QQ}\right),
    \end{align*}
which appears on both in Table \ref{tab:semi_euclidean_dim_4} and Table \ref{tab:semi_euclidean_dim_5} is not even an Eichler order, must less hereditary.
\end{remark}

\begin{problem}
If $R$ is a semi-Euclidean ring, does there exist an involution $\ddagger$ on $R$ such that $(R,\ddagger)$ is a $\ddagger$-Euclidean ring?
\end{problem}

\begin{remark}
While it seems highly unlikely that this is true in general, it is true for all of the orders that we have constructed: if $\OO$ is an order of an imaginary quadratic field $K$, then if we take $\ddagger$ to be complex conjugation, one checks that $\OO^+ = \ZZ$---therefore, if $\OO$ is semi-Euclidean, it is certainly $\ddagger$-Euclidean. This handles the $\Dim = 3$ case. Less trivially, in the $\Dim = 5$ case, we observe that the three semi-Euclidean (but not Euclidean) orders
    \begin{align*}
        \ZZ \oplus \ZZ i \oplus \ZZ j \oplus \ZZ ij &\subset \left(\frac{-1,-1}{\QQ}\right) \\
        \ZZ \oplus 3\ZZ i \oplus \ZZ (i - j) \oplus \ZZ \frac{1 + i + j + ij}{2} &\subset \left(\frac{-1,-1}{\QQ}\right) \\
        \ZZ \oplus \ZZ i \oplus \ZZ j \oplus \ZZ \frac{1 + i + j + ij}{2} &\subset \left(\frac{-3,-1}{\QQ}\right)
    \end{align*}
are isomorphic to the three $\ddagger$-Euclidean orders
    \begin{align*}
        \ZZ \oplus \ZZ i \oplus \ZZ j \oplus \ZZ ij &\subset \left(\frac{-1,-1}{\QQ}\right) \\
        \ZZ \oplus \ZZ i \oplus \ZZ \frac{1 + j}{2} \oplus \ZZ \frac{i + ij}{2} &\subset \left(\frac{-2,-3}{\QQ}\right) \\
        \ZZ \oplus \ZZ i \oplus \ZZ \frac{1 + i + j}{2} \oplus \ZZ \frac{1 + i + ij}{2} &\subset \left(\frac{-1,-6}{\QQ}\right).
    \end{align*}
    
\noindent That the Euclidean orders in the $\Dim = 5$ case are $\ddagger$-Euclidean was already observed\cite{Sheydvasser_2021}.
\end{remark}

\begin{problem}
What interesting examples of semi-Euclidean and semi-$\ddagger$-Euclidean rings are there? Are there infinitely many isomorphism classes among orders of quaternion algebras over number fields? Are there non-commutative examples where the stathm is not multiplicative?
\end{problem}

\begin{remark}
This is essentially an extension of problems 4-6 listed in a prior paper of the author\cite{Sheydvasser_2021}. 
\end{remark}

\begin{problem}
Let $\OO$ be a maximal order of a central simple algebra $A$ over a number field $K$. If $\OO$ is semi-Euclidean, is $\OO$ Euclidean? Similarly, if $\OO$ is a maximal $\ddagger$-order of a CSA with involution $(A,\ddagger)$ over a number field $K$, does $\OO$ being semi-$\ddagger$-Euclidean imply that it is $\ddagger$-Euclidean?
\end{problem}

\begin{remark}
This is true of all of the orders we constructed.
\end{remark}

\subsection*{Acknowledgment.}
The author would like to thank Nir Lazarovich for providing the final nudge that allowed adding ``non-normal" to ``infinite-index" in the statement of the main theorem.

\bibliography{References}
\bibliographystyle{plain}

\end{document}